\documentclass[12pt]{amsart}
\usepackage{graphicx} 

\usepackage[top=1in, bottom=1.25in, left=1in, right=1in]{geometry}
\usepackage{amssymb}
\usepackage{amsmath}
\usepackage{hyperref}
\newtheorem{theorem}{Theorem}[section]
\newtheorem{lemma}[theorem]{Lemma}
\newtheorem{prop}[theorem]{Proposition}

\newtheorem{cor}[theorem]{Corollary}

\theoremstyle{definition}
\newtheorem{defi}[theorem]{Definition}
\theoremstyle{remark}

\DeclareMathOperator{\diam}{diam}%
\newcommand{\map}[3]{#1\colon #2\to #3}
\newcommand{\field}[1]{\mathbb{#1}}
\newcommand{\rr}{\field{R}}
\newcommand{\zz}{\field{Z}}
\newcommand{\st}{\;|\;}
\newcommand{\kre}[1]{\overline{#1}}
\newcommand{\eps}{\varepsilon}
\newcommand{\dpn}{d_{PN}}
\newcommand{\dps}{d_{PS}}
\newcommand{\dwp}{d_{WP}}
\newcommand{\cX}{\mathcal{X}}
\newcommand{\cY}{\mathcal{Y}}
\newcommand{\Sig}{\Sigma}
\newcommand{\mcg}{\mathrm{Mod}}
\newcommand{\pants}{\mathrm{Pants}}
\newcommand{\pantsV}{\mathrm{Pants^0}}
\newcommand{\pantsN}{\pants(N)}
\newcommand{\pantsNV}{\pantsV(N)}

\newcommand{\pantsS}{\pants(S)}
\newcommand{\pantsSV}{\pantsV(S)}
\newcommand{\pantsSg}{\pantsS}
\newcommand{\pantsSig}{\pants^\sigma(S)}
\newcommand{\pantsSgSig}{\pantsSig}
\newcommand{\Teich}{\mathrm{Teich}}
\newcommand{\TS}{\Teich(S)}
\newcommand{\TN}{\Teich(N)}
\newcommand{\TSg}{\TS}
\newcommand{\TsSig}{\Teich^{\sigma}(S)}
\newcommand{\TsgSig}{\TsSig}

\newcommand{\Q}{\Phi}
\newcommand{\J}{j}
\newlength{\customwidth}
\newcommand{\mathdash}{\relbar\mkern-4mu\relbar}

\newcommand{\plainMove}[2]{\,\ooalign{$#2$\cr
    \hidewidth\raisebox{1.15ex}{\text{\tiny {\ #1}}} 
    \hidewidth\cr}\,}
 
\newcommand{\mvI}{\plainMove{\;1}{\mathdash}}  
  
\newcommand{\mvIII}{\plainMove{\,3}{\mathdash}}  
\newcommand{\mvIV}{\plainMove{4}{\mathdash}}

\begin{document}
\title[The pants graph as a combinatorial model\ldots]{The pants graph as a combinatorial model for the Teichm\"uller space of a non-orientable surface}
\author{Micha{\l}\ Stukow, B{\l}a{\.z}ej\ Szepietowski, Jakub Szmelter-Tomczuk}
\address{Institute of Mathematics, Faculty of Mathematics, Physics and Informatics, University of Gda\'nsk, 80-308 Gda\'nsk, Poland}
\email{michal.stukow@ug.edu.pl}
\email{blazej.szepietowski@ug.edu.pl}
\email{jakub.szmelter-tomczuk@ug.edu.pl}


\numberwithin{equation}{section}

\begin{abstract}
We study the relation between the pants graph of a non-orientable surface and that of its orientable double cover. 
Given a non-orientable surface $N$ with orientable double cover $\pi\colon S\to N$, we construct a natural map between pants graphs induced by lifting pants decompositions. We prove that this map defines a quasi-isometric embedding of $\pantsN$ into $\pantsS$. Using Brock’s quasi-isometry between $\pantsS$ and Teichmüller space $\TS$ endowed with the Weil–Petersson metric, together with the identification of the Teichmüller space of $N$ with the fixed-point locus of the deck involution on $\TS$, we prove that $\pantsN$ is quasi-isometric to the Teichmüller space of $N$ equipped with the induced Weil–Petersson metric.
\end{abstract}

\maketitle

\section{Introduction}
The pants complex of an orientable surface was introduced by Hatcher and Thurston~\cite{HatcherThurston1980}, and its structure has since been extensively studied in connection with Teichmüller spaces and mapping class groups. The $1$-skeleton of the pants complex is the pants graph, whose vertices are pants decompositions and whose edges correspond to elementary moves. 
Brock  \cite{Brock2003} showed that the pants graph is quasi-isometric to the Teichmüller space equipped with the Weil–Petersson metric, while Margalit \cite{Margalit2004} proved that its automorphism group is naturally isomorphic to the extended mapping class group. In the non-orientable setting, Papadopoulos and Penner \cite{Papadopoulus2016} introduced an analogous pants graph and established several of its basic properties.  The purpose of the present paper is to investigate the large-scale geometry of pants graphs of non-orientable surfaces and their relation to the Weil–Petersson geometry of the corresponding Teichm\"uller spaces.

Throughout, $S_{g,n}$ (respectively $N_{g,n}$) denotes an orientable (respectively non-orientable) topological surface of genus $g$ with $n$ punctures. Let $\pi\colon S\to N$ be the orientable double cover of $N=N_{g,n}$, where $S=S_{g-1,2n}$ and let $\sigma$ denote the orientation-reversing deck involution of $S$. We denote by $\pantsN$ and $\pantsS$ the  pants graphs of $N$ and $S$, equipped with their graph metrics $\dpn$ and $\dps$, respectively. Our first main result establishes a large-scale geometric relationship between  these two metric spaces.
\begin{theorem}\label{main_pantsN_pantsS}
The orientable double cover $\pi\colon S\to N$ induces
a quasi-geometric embedding 
\[\map{\J}{(\pantsNV,\dpn)}{(\pantsSV,\dps)},\quad \J(P)=\pi^{-1}(P),\]
where $\pants^0(\cdot)$ denotes the vertex set of the corresponding pants graph.
\end{theorem}
The image of $\J$ consists precisely of the $\sigma$-invariant pants decompositions of $S$ and will be denoted by 
$\pantsSig$. Endowed with the metric induced by $\dps$, the space $\pantsSig$ is quasi-isometric to $\pantsN$. We emphasize, however, that $\pantsSig$ is not itself a subgraph of $\pantsS$, and therefore the induced metric is not a graph metric in the usual sense.

Theorem \ref{main_pantsN_pantsS} should be compared with analogous results for curve complexes and mapping class groups. Masur and Schleimer \cite{MasurShleimer2013} established a corresponding relation between the curve complexes of non-orientable surfaces and their orientable double covers, while Katayama and Kuno \cite{KatayamaKuno2024} obtained similar comparison at the level of mapping class groups.

The proof of Theorem \ref{main_pantsN_pantsS} splits naturally into two parts. For $P,P'\in\pantsNV$ we seek linear upper and lower bounds for $\dps(\J(P),\J(P'))$ in terms of $\dpn(P,P')$. 

The upper bound follows from a direct analysis of elementary moves: if $P$ and $P'$ differ by an elementary move in $\pantsN$, then their lifts can be connected in $\pantsS$ by at most three 
moves. This immediately implies that $\J$ is $3$-Lipschitz.

The existence of the lower bound is considerably more subtle. Indeed, it implies that every geodesic path in $\pantsS$ with endpoints in $\pants^\sigma(S)$ can be replaced by a sequence of symmetric vertices with the same endpoints, whose length is coarsely proportional to that of the original geodesic and whose projections form a path in $\pantsN$. Note, however, that such a sequence does not in general define a path in $\pantsS$,


Our proof of the lower bound relies on the Weil–Petersson geometry of Teich\"uller space $\TS$ and Brock's quasi-isometry theorem \cite{Brock2003}. The involution $\sigma$ induces an isometric action on $\TS$, whose fixed-point set $\TsSig$ is naturally identified with $\TN$. Following Brock's construction, and incorporating a result of Sepp\"al\"a \cite{Seppala1991}, we  define a map \[\Q\colon\pants^\sigma(S)\to\Teich^\sigma(S),\]
and we show that the composition $\Q\circ j$ is a quasi-isometry. This implies that $j$ is indeed a quasi-isometric embedding and it yields our second main theorem.
\begin{theorem}\label{main_pantsN_Teich}
The pants graph $\pants(N)$ is quasi-isometric to $\Teich^\sigma(S)$ endowed with the induced Weil-Petersson metric.      
\end{theorem}

Since $\TsSig$ is naturally identified with $\TN$, Theorem \ref{main_pantsN_Teich} provides a direct analogue of Brock’s quasi-isometry theorem in the non-orientable setting. In particular, it shows that $\pantsN$ serves as a coarse combinatorial model for $\TN$ endowed with the metric induced from the Weil–Petersson metric on $\TS$. 

A remark is in order regarding Brock’s theorem: although the result in \cite{Brock2003} is stated for closed surfaces, the proof extends without essential modification to punctured surfaces, and this generality is used here.

Since $\pants(N_{1,2})$ consists of two vertices connected by an edge, and $\pants(N_{2,1})$ is a graph of diameter 2 (see \cite{Papadopoulus2016} or \cite{stukow2025}), as an immediate application of Theorem~\ref{main_pantsN_Teich}, we obtain the following corollary.
\begin{cor}
    The Teichmüller spaces of $N_{1,2}$ and $N_{2,1}$, equipped with the induced Weil-Petersson metric, are each quasi-isometric to a point.
\end{cor}


We conclude the introduction with a related open question motivated in part by the results of the present paper. Masur and Wolf \cite{MasurWolf2002} proved that the Weil–Petersson isometry group of $\TS$ is the extended mapping class group $\mcg(S)$. Brock and Margalit \cite{BrockMargalit2007} gave an alternative proof using Margalit’s description of automorphisms of the pants graph~\cite{KatayamaKuno2024}. Recently, the first two authors \cite{stukow2025} proved that the automorphism group of the pants graph of a non-orientable surface $N$ is isomorphic to its mapping class group $\mcg(N)$. Moreover, by a theorem of Birman and Chillingworth \cite{BirmanChillingworth1972}, $\mcg(N)$ is isomorphic to the quotient $\mcg^\sigma(S)/\langle\sigma\rangle$, where $\mcg^\sigma(S)$ is the centraliser of $\sigma$ in $\mcg(S)$.
This leads naturally to the following open question.

\begin{quote}
Is the Weil–Petersson isometry group of $\TsSig$ isomorphic to $\mcg(N)$?    
\end{quote}

The paper is organized as follows. In Section~2 we recall the necessary background on pants graphs and Teichm\"uller spaces. In Section~3 we define the map $\J$ and prove the $3$-Lipschitz bound. Section~4 introduces symmetric sublevel subsets of $\TsSig$, which are used in Section~5 to construct the map $\Q$ and prove Theorems~\ref{main_pantsN_pantsS} and~\ref{main_pantsN_Teich}.
%
\section{Preliminaries}
Let $(\cX,d_\cX)$,  $(\cY,d_\cY)$ be metric spaces. A map $\map{f}{\cX}{\cY}$ is a {\it quasi-isometric embedding} if there exits constants $A_1\ge 1$ and $A_2\ge 0$ such that 
\[\frac{1}{A_1}d_\cX(x,y)-A_2\le d_\cY(f(x),f(y))\le A_1d_\cX(x,y)+A_2\]
for all $x,y\in\cX$. If, in addition, $f(\cX)$ is $A_2$-{\it dense} in $\cY$ (every point of $\cY$ lies within distance $A_2$ of $f(\cX)$), then we say that $f$ is a {\it quasi-isometry}. 

Let $\Sig\in\{S_{g,n}, N_{g,n}\}$ be a surface of negative Euler characteristic. Recall that $\chi(S_{g,n})=2-2g-n$, whereas  $\chi(N_{g,b})=2-g-n$.  The {\it mapping class group} $\mcg(\Sig)$ of $\Sig$ is the group of isotopy classes
of all homeomorphisms of $\Sig$. If $\Sig$ is orientable, then $\mcg(\Sig)$ is usually called the extended
mapping class group.

By a {\it non-trivial curve} or simply a {\it curve} on $\Sig$ we understand an embedded simple closed curve that does not bound a disk with at most one puncture or a M\"obius band.
A curve on $\Sig$ is said to be \emph{two-sided} (resp. \emph{one-sided}) if its regular neighborhood is an annulus (resp. a M\"obius band). Abusing the language, we shall identify a curve on $\Sig$ with its isotopy class. By $i(\alpha,\beta)$ we denote the geometric intersection number of curves $\alpha$ and $\beta$. We say that $\alpha$ and $\beta$ are {\it disjoint} when $\alpha \ne \beta$ and $i(\alpha,\beta)=0$. A set of pairwise disjoint and non isotopic curves on $\Sig$ is called a {\it multicurve}.  For a multicurve $M$ on $\Sig$, we denote by $\Sig\setminus M$ the complement in $\Sig$ of the union of the elements of $M$. A connected component of $\Sig\setminus M$ is  {\it trivial} if it is homeomorphic to a pair of pants, i.e. a three-holed sphere. If all components of $\Sigma\setminus M$ are trivial, then $M$ is a {\it pants decomposition} of $\Sig$.

\subsection{The pants graph.}
 We say that two pants decompositions $P$ and $P'$ of $\Sig$ differ by an {\it elementary move} if $P'$ is obtained from $P$ by replacing a single curve $\alpha\in P$ by either a single curve $\beta\in P'$ or by a pair of disjoint one-sided curves $\beta_1, \beta_2\in P'$, in such a way that the geometric intersection number of $\alpha$ and $\beta$ (resp. $\alpha$ and  $\beta_1\cup\beta_2$) is minimal. The unique nontrivial component of $\Sig\setminus(P\setminus\{\alpha\})$ is called the  {\it support} of the elementary move. It is homeomorphic to one of the surfaces: $S_{1,1}$, $S_{0,4}$, $N_{1,2}$, $N_{2,1}$, which gives us four types of elementary moves -- see Figure \ref{fig:moves:all}, where the shaded discs represent cross-caps. Note that the support of each move can be identified with the regular neighbourhood $N(\alpha\cup \beta)$ or $N(\alpha\cup \beta_1\cup \beta_2)$ respectively.
\begin{figure}[h]
\begin{center}
\includegraphics[width=0.99\customwidth]{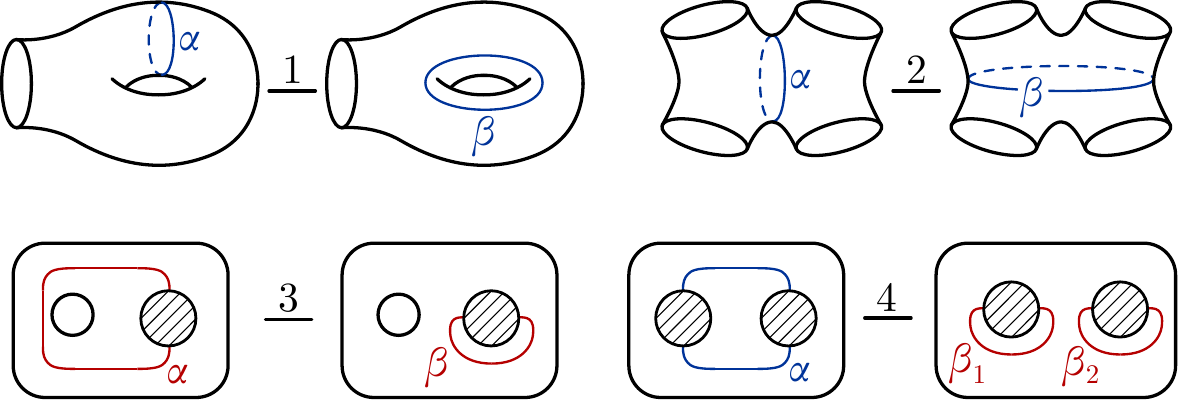}
\caption{The four types of elementary moves between pants decompositions.}\label{fig:moves:all} %
\end{center}
\end{figure}
The \emph{pants graph} $\pants(\Sig)$ of $\Sig$ is a graph whose vertex set $\pantsV(\Sig)$ is the set of all pants decompositions of $\Sig$, and whose edges correspond to elementary moves. 

The pants graph of an orientable surface is defined using only moves of type 1 and 2. Moves of type 3 and 4 are supported on non-orientable surfaces and involve one-sided curves. Note that a move of type 4 changes the number of curves in a pants decomposition, which reflects the fact that this number varies if the surface is non-orientable. More precisely, for every integer $m$ such that $0\le m\le g$ and $m\equiv g\pmod{2}$ there exits a pants decomposition of $N_{g,n}$ of cardinality 
\[\frac{3g+m}{2}+n-3\] containing exactly $m$ one-sided curves. It was shown in \cite{HatcherLochakSchneps} and \cite{Papadopoulus2016} that for every surface~$\Sig$, the corresponding pants graph $\pants(\Sig)$ is connected.

\subsection{The Weil-Petersson metric on the Teichm\"uller space} 
A point in the  {\it Teichm\"uller space} $\Teich(\Sigma)$ of $\Sigma$ is given by a pair
$(X,\varphi)$, where $X$ is a complete finite-area hyperbolic surface  and $\varphi\colon \Sigma \to X$ is a homeomorphism. Two points $(X_1,\varphi_1)$ and $(X_2,\varphi_2)$ are equivalent if there exists an isometry $\psi\colon X_1\to X_2$ such that $\psi \circ \varphi_1$ is isotopic to $\varphi_2$. 
There is a natural action of $\mcg(\Sigma)$ on $\Teich(\Sigma)$ defined by
$f\cdot(X,\varphi)=(X,\varphi\circ f^{-1})$ for $f\in\mcg(\Sigma)$.

We refer the reader to  \cite{Brock2003} for the definition of the Weil–Petersson metric on the Teichm\"uller space of an orientable surface. For non-orientable surfaces, the Weil–Petersson metric can be defined via the orientable double cover.
Let $\map{\pi}{S}{N}$ be the orientable double cover, with $S=S_{g-1,2n}$ and $N=N_{g,n}$, and let $\sigma$ denote the orientation-reversing deck involution. The action of $\sigma$ on $S$ induces an involution of  $\TS$, whose fixed-point locus, denoted by $\TsSig$, is naturally identified with $\TN$ (see Chapter 4 of \cite{SeppalaSorvali1992}). We equip $\TN$ with the Weil–Petersson metric induced from $\TS$, namely the restriction of the Weil–Petersson metric $\dwp$ on $\TS$ to the submanifold $\TsSig$.
%
\begin{prop}\label{Ts:sigma:convex} 
$\TsSig$ is a convex subset of $\TS$.
\end{prop}
\begin{proof}
If $X,Y\in \TsSig$ and $X\neq Y$, then by Wolpert's result (Remark 5.4 and Corollary~5.9 of \cite{Wolpert1987}), $X$ and $Y$ are connected by a unique geodesic segment $[X,Y]$ in $\TS$. The extended mapping class group $\mcg(S)$ acts via Weil-Petersson isometries on $\TS$. In particular, $\sigma$ induces an isometry of $\TS$.
Since $X$ and $Y$ are fixed by $\sigma$, the same is true for the geodesic segment $[X,Y]$.
\end{proof}
\section{Lifting pants decompositions}
In this section we use the orientable double cover $\map{\pi}{S}{N}$ to define a map
\[\map{\J}{\pantsNV}{\pantsSig},\]
where $\pantsSig$ is the subset of $\pantsSV$ consisting of $\sigma$-invariant pants decompositions of $S$.

\begin{lemma}\label{lem:pi:lift}
    If $T$ is an oriented and connected subsurface of $N$, then $\pi^{-1}(T)$ has exactly two connected components $M$ and $\sigma(M)$, and 
    \[\map{\pi_{|M}}{M}{T}\] 
    is a homeomorphism.
\end{lemma}
\begin{proof}
    Since $\pi$ has the path lifting property, $\pi^{-1}(T)$ has at most two connected components: $M$ and $\sigma(M)$. Moreover, 
    \[\map{\pi_{|M}}{M}{T}\] 
    is injective. In fact, if $m_1,m_2\in M$ are such that $m_1\neq m_2$ and $\pi(m_1)=\pi(m_2)$, then any path connecting $m_1$ and $m_2$ in $M$ projects to a one-sided loop in $T$, which is not possible. Hence, $M\neq \sigma(M)$.
\end{proof}
\begin{lemma}\label{J:bijection}
    If $P$ is a $\sigma$-invariant pants decomposition of $S$, then $\pi(P)$ is a pants decomposition of $N$.
 If $Q$ is a  pants decomposition of $N$, then $\pi^{-1}(Q)$ is a $\sigma$-invariant pants decomposition of $S$.   
\end{lemma}
\begin{proof}
    Since $\pi$ is a local homeomorphism, every component of $\pi(P)$ or $\pi^{-1}(Q)$ is a simple closed curve.
    It remains to show that every component of $N\setminus\pi(P)$ or $S\setminus\pi^{-1}(Q)$ is a pair of pants.

    Let $M$ be a component of $S\setminus P$. We claim that $M$ is disjoint from $\sigma(M)$. Otherwise $\sigma(M)=M$ and $\pi(M)$ is a non-orientable surface of Euler characteristic $\frac{1}{2}\chi(M)=-\frac{1}{2}$, which is impossible. It follows that $\pi|_M$ is a homeomorphism and $\pi(M)$ is a pair of pants.

    Now let $T$ be a component of $N\setminus Q$. By Lemma \ref{lem:pi:lift}, $\pi^{-1}(T)$ has two components, say $M$ and $\sigma(M)$. Moreover, $M$ is a component of $S\setminus\pi^{-1}(Q)$ and is homeomorphic to $T$. It follows that $M$ is a pair of pants.
\end{proof}
Recall that if $Q\in\pants^0(N_{g,n})$ contains $m$ one-sided curves, then
\[|Q|=\frac{3g+m}{2}+n-3.\]
Since every one-sided curve on $N$ lifts to a single curve on $S$, whereas every two-sided curve lifts to a pair of curves, we have
\[|\pi^{-1}(Q)|=2|Q|-m=3(g-1)+2n-3,\]
which is exactly the number of curves in a pants decomposition of $S_{g-1,2n}$.
    


We denote by $\dpn$ and $\dps$ the graph metrics on $\pantsN$ and $\pantsS$ respectively.
The subset $\pantsSig\subset\pantsSV$ is endowed with the metric induced by $\dps$. 

\begin{prop}\label{J:3Lip}
The map 
\[\map{\J}{\pantsNV}{\pantsSig}\]
defined by 
$\J(P)=\pi^{-1}(P)$ is a well defined bijection. Moreover, 
$\J$ is 3-Lipschitz with respect to the metrics $\dpn$ on $\pantsN$ and $\dps$ on $\pantsSig$.
\end{prop}
\begin{proof}
The first part of the statement follows from Lemma \ref{J:bijection}. 

Let $P,P'\in \pantsN$ be two pants decompositions such that $\dpn(P,P')=1$. It is enough to prove that $\dps(j(P),j(P'))\leq 3$.

If $P$ and $P'$ differ by a move of type 1 (see Figure \ref{fig:moves:all})
\[\alpha \mvI \beta, \]
then, by Lemma \ref{lem:pi:lift}, the support $N(\alpha\cup\beta)\simeq S_{1,1}$ of this move lifts to two disjoint copies 
\[N\left(\kre{\alpha}_1\cup \kre{\beta}_1\right)\amalg N\left(\kre{\alpha}_2\cup \kre{\beta}_2\right)\]
of $N(\alpha\cup\beta)$. 
Therefore, we can connect $\J(P)$ and $\J(P')$ by a path of length 2 in $\pantsS$:
\[\kre{\alpha}_1,\kre{\beta}_1\mvI \kre{\alpha}_1,\kre{\beta}_2\mvI \kre{\alpha}_2,\kre{\beta}_2.\]
The situation is completely analogous if $P$ and $P'$ differ by a move of type 2 -- in this case the path connecting $\J(P)$ and $\J(P')$ consists of two moves of type 2.

If $P$ and $P'$ are connected by a move of type 3
\[\alpha \mvIII \beta, \]
then the lift of the support $N(\alpha\cup\beta)\simeq N_{1,2}$ of this move is the surface $N(\kre{\alpha}\cup \kre{\beta})$ homeomorphic to a 4-holed sphere $S_{0,4}$ -- see Figure \ref{fig:lift:mv3}. 
    \begin{figure}[h]
\begin{center}
\includegraphics[width=0.49\customwidth]{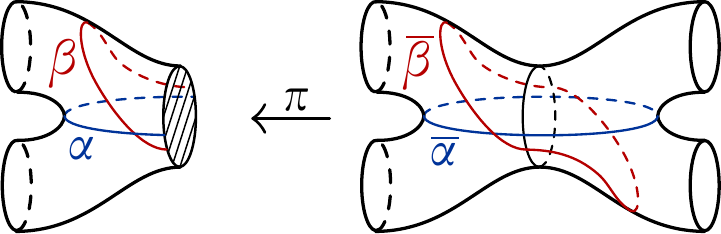}
\caption{The lift of an elementary move of type 3 as a move of type 2 in $\pantsS$.}\label{fig:lift:mv3} %
\end{center}
\end{figure}
In this case $\J(P)$ and $\J(P')$ differ by a single move of type 2.

Finally, if $P$ and $P'$ are connected by a move of type 4
\[\alpha \mvIV \beta_1,\beta_2, \]
then the lift of the support $N(\alpha\cup\beta_1\cup \beta_2)\simeq N_{2,1}$ is the surface $N(\kre{\alpha}_1\cup\kre{\alpha}_2\cup\kre{\beta}_1\cup\kre{\beta}_2)$ homeomorphic to a 2-holed tours $S_{1,2}$. In this case $\J(P)$ and $\J(P')$ can be connected by path of length 3 -- see Figure \ref{fig:lift:mv4}. 
\begin{figure}[h]
\begin{center}
\includegraphics[width=0.63\customwidth]{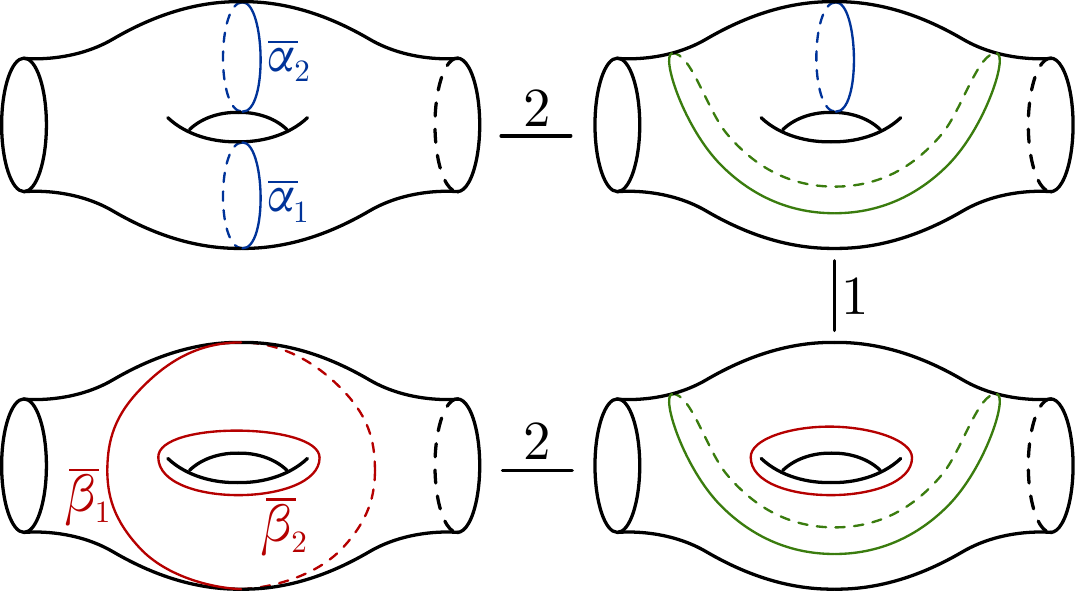}
\caption{The lift of an elementary move of type 4 as a path of length 3 in $\pantsS$.}\label{fig:lift:mv4} %
\end{center}
\end{figure}
\end{proof}

\section{Symmetric sublevel sets}

As in \cite{Brock2003}, for each pants decomposition $P\in \pantsSg$ and $\ell\in\rr_+$, we associate the \emph{sublevel set}
\[V_{\ell}(P)=\left\{ X\in \TSg \st \max_{\alpha\in P}\{\ell_{X}(\alpha)\}<\ell \right\},\]
where $\ell_{X}(\alpha)$ denotes the length of the geodesic representative of $\alpha$ in the hyperbolic metric on $X$.
By a theorem of Bers, there exists a constant $B$, depending only  on the topological type of $S$, such that 
the collection $\left\{V_{B}(P)\right\}_{P\in\pantsSg}$ forms a cover of $\TSg$.

We require an analogous construction for the subspace $\TsgSig\subset \TSg$ of symmetric hyperbolic metrics on $S$. Instead of a  Bers constant, we use the constant $L$ provided by the following theorem.
\begin{theorem}[Theorem 4.8 of \cite{Seppala1991}] There exits a constant $L>0$, depending only on the topological type of $S$, such that for every $X\in \TsgSig$ there exits a pants decomposition $P\in \pantsSgSig$
satisfying $\ell_X(\alpha)<L$ for every $\alpha\in P$. \label{Bers-const}
\end{theorem}
The above theorem is stated in \cite{Seppala1991} for closed surfaces, but its proof extends without essential changes to punctured surfaces.
By increasing $L$, if necessary, we may assume $L\ge B$, where $B$ is a Bers contant for $S$. 
Throughout the paper, $L$ will always denote such a constant.
\begin{defi} If $\ell\in\rr_+$ and $P\in \pantsSgSig$, then we define the \emph{symmetric sublevel set}
\[V_{\ell}^\sigma(P)=V_{\ell}(P)\cap \TsgSig.\]
\end{defi}
\begin{prop}\label{sublevel:prop}
 If $\ell\in\rr_+$ and $P\in \pantsSgSig$, then 
 \begin{enumerate}
  \item each $V_{\ell}^\sigma(P)$ is convex in the Weil-Petersson metric, and
  \item there is a constant $D(\ell)>0$, depending only on $S$ and $\ell$, for which
  \[\diam_{WP}(V_{\ell}^\sigma(P))<D(\ell). \]
 \end{enumerate}
 \end{prop}
 \begin{proof}
By Proposition 2.2 of \cite{Brock2003}, the set $V_{\ell}(P)$ is convex in the Weil-Petersson metric, and there exists a constant $D(\ell)$, depending only on $S$ and $\ell$, such that 
\[\diam_{WP}(V_{\ell}(P))<D(\ell).\]
Since
\[V_{\ell}^\sigma(P)=V_{\ell}(P)\cap \TsgSig,\]
we have
\[\diam_{WP}(V_{\ell}^\sigma(P))\le\diam_{WP}(V_{\ell}(P))<D(\ell), \]
which proves (2). By Proposition~\ref{Ts:sigma:convex},
$\TsgSig$ is convex in the Weil--Petersson metric. Hence
$V_{\ell}^\sigma(P)$, being the intersection of two convex sets, is convex as
well, proving (1).
 \end{proof}
%
\begin{lemma}\label{lem:a}
For every $L'>0$, there exists an integer $K=K(L')>0$, depending only on $S$ and $L'$, with the following property:
if $P,P'\in\pantsNV$ satisfy
\[V_{L'}(\J(P))\cap V_{L'}(\J(P'))\neq \emptyset,\]
then $\dpn(P,P')\leqslant K$.
\end{lemma}
\begin{proof}
 We adapt the proof of Lemma 3.3 in \cite{Brock2003}.

 If $V_{L'}(\J(P))\cap V_{L'}(\J(P'))\neq \emptyset$, then there exists $X\in\TS$ such that $\ell_X(\alpha)<L'$ for each $\alpha\in P\cup P'$. Then, by 
 the collar lemma (Theorem 4.4.6 of \cite{BuserBook}), there exists $\eps>0$ such that the geodesic realization of $P$ on $X$ has an embedded metric $\eps$-collar on $X$. This implies that there exists a constant $C$ bounding the total intersection number
 \[i(\J(P),\J(P'))\leq C.\]
 The same bound holds in $N$
 \begin{equation}\label{eq:i:bound}
  i(P,P')\leq C.
 \end{equation}
Now we consider the action 
\[T(P)\times \pantsNV\to \pantsNV\]
of the group $T(P)$ generated by Dehn twists about two-sided curves in $P$. Functions
\begin{align*}
\map{i(P,\cdot)}{\pantsNV}{\zz}\\
\map{\dpn(P,\cdot)}{\pantsNV}{\zz}
\end{align*}
factor through this action and descend to functions
\begin{align*}
\map{i(P,\cdot)}{\pantsNV/T(P)}{\zz}\\
\map{\dpn(P,\cdot)}{\pantsNV/T(P)}{\zz}.
\end{align*}
There is a finite set
\[[P_1],\ldots,[P_c]\]
of equivalence classes in $\pantsNV/T(P)$ satisfying $i(P,[P_k])\leq C$ and by \eqref{eq:i:bound}, we are interested only in such classes. So we take 
\[K_P=\max_{1\leq k\leq c}\dpn(P,[P_k]),\]
and then take $K$ to be the maximum of $K_P$ over all possible topological types of $P$.
\end{proof}
\begin{lemma}\label{lem:b}
Given $L'>L$, where $L$ is the constant from Theorem \ref{Bers-const}, there exists an integer $J=J(L')>0$, depending only on $S$ and $L'$, with the following property: if $X_t$, $t\in [0,1]$, is a Weil-Petersson geodesic segment of length $1$  joining $X_0, X_1\in \TsgSig$, then there exist  symmetric pants decompositions $P_1,\ldots, P_J\in \pantsSgSig$ such  that
\[\{X_t\st t\in[0,1]\}\subset V_{L'}(P_1)\cup\ldots\cup V_{L'}(P_J).\]
\end{lemma}
\begin{proof}
    The statement is a minor reformulation of Lemma 3.4 in \cite{Brock2003}, so we only sketch the proof -- for details see \cite{Brock2003}.

    By Theorem \ref{Bers-const}, 
\begin{equation}\label{lem34:cover}
  \{X_t\st t\in[0,1]\}\subset V_L(P_1)\cup\ldots\cup V_L(P_m)  
\end{equation}
for some $P_1,\ldots,P_m\in \pantsSgSig$. It is shown in \cite{Brock2003}, that there exists a constant $\eps$, depending only on $S$, $L$ and $L'$, such that if $X_{t_0}\in V_L(P)$, then 
\begin{equation}\label{lem34:eqEps}
X_t\in V_{L'}(P),\quad \text{for $t\in (t_0-\eps,t_0+\eps)$}.    
\end{equation}
Now, take $J=\left\lceil \frac{2}{\eps_0}\right\rceil$, and for each $a_i=i\cdot \frac{\eps}{2}$, $i=0,\ldots, J$, let $P_{k(i)}$ be such an element of the cover \eqref{lem34:cover} that $X_{a_i}\in P_{k(i)}$. By \eqref{lem34:eqEps}, 
\[\{X_t\st t\in[0,1]\}\subset V_{L'}(P_{k(1)})\cup\ldots\cup V_{L'}(P_{k(J)}). \]
\end{proof}
\section{Proofs of the main theorems}
We define a map
\[\map{\Q}{\pantsSgSig}{\TsgSig},\]
by the condition
\[\Q(P)\in V^\sigma_L(P),\]
where $L$ is the constant from Theorem \ref{Bers-const}. We will show that the composition
\[\map{\Q\circ \J}{\pantsNV}{\TsgSig}\]
is a quasi-isometry, establishing Theorem \ref{main_pantsN_Teich}.

\begin{lemma}\label{Lipschitz6d}
    The map $\Q\circ\J$ 
is $6D$-Lipschitz, and its image is $D$-dense in $\TsgSig$, where $D=D(L)$ is the constant from Proposition \ref{sublevel:prop}.
\end{lemma}
\begin{proof}
    By Proposition \ref{J:3Lip}, the map 
    \[\map{\J}{\pantsNV}{\pantsSgSig}\]
    is 3-Lipschitz, and $\Q$
    is a restriction of the map
    \[\map{\kre{\Q}}{\pantsSV}{\TSg}\]
    defined by the condition
    \[\kre{\Q}(P)\in V_L(P).\]
    Brock proved in \cite{Brock2003} that $\kre{\Q}$ is $2D$-Lipschitz, hence $\Q\circ \J$ is $6D$-Lipschitz.

    By Theorem \ref{Bers-const} and surjectivity of $\J$,
\[\TsgSig = \sum_{P\in \pantsSgSig}V^\sigma_L(P)=\sum_{Q\in \pantsNV}V^\sigma_L(j(Q)),\]
and by Proposition \ref{sublevel:prop}, $\diam_{WP}(V^\sigma_L(j(Q)))<D$ for each $Q\in\pantsNV$. Since each $V^\sigma_L(j(Q))$ intersects the image of~$\Q\circ\J$, this image is $D$-dense.
\end{proof}

Now we are ready to prove Theorems \ref{main_pantsN_Teich} and \ref{main_pantsN_pantsS}.
\begin{proof}[Proof of Theorem \ref{main_pantsN_Teich}]
We will show that 
 \[\map{\Q\circ \J}{\pantsNV}{\TsgSig}\]
 is a quasi-isometry. By Lemma \ref{Lipschitz6d}, the image of $\Q\circ \J$
  is $D$-dense in $\TsgSig$, so it is enough to prove that $\Q\circ \J$ is a quasi-isometric embedding. Therefore, we need to find two constants $A_1\ge 1$ and $A_2\ge 0$ such that 
\begin{equation}\label{qi:embedding}
    \frac{1}{A_1}\cdot \dpn(P,P')-A_2\leq \dwp(\Q\circ \J(P),\Q\circ \J(P'))\leq 
    A_1\cdot \dpn(P,P')+A_2
\end{equation}
for any $P,P'\in\pantsNV$.
Lemma \ref{Lipschitz6d} implies that the second inequality of \eqref{qi:embedding} is satisfied for any $A_2\ge 0$ and $A_1\ge 6D$.

For the first inequality, we follow the proof of Theorem 3.2 in \cite{Brock2003}. If $X=\Q\circ \J(P)$ and $X'=\Q\circ \J(P')$, then Proposition \ref{J:3Lip} and Lemma \ref{lem:b} imply that there exists a sequence $\{P_i\}_{i=0}^{m}\in\pantsNV$ of pants decompositions such that
\begin{enumerate}
    \item $P_0=P$, $P_m=P'$ and $\{V_{2L}(j(P_i))\}_{i=0}^{m}$ covers the geodesic segment joining $X$ and $X'$ in $\TS$;
    \item $V_{2L}(\J(P_i))\cap V_{2L}(\J(P_{i+1}))\neq \emptyset$, for $i=0,\ldots,m-1$;
    \item if $J=J(2L)$ is a constant provided by Lemma \ref{lem:b}, then 
    \begin{equation}\label{eq:main:proof:A}
    m\le \left(\dwp(X,X')+1\right)\cdot J.
    \end{equation}
\end{enumerate}
Lemma \ref{lem:a}, (1) and (2) imply that 
\begin{equation}\label{eq:main:proof:B}
  \dpn(P,P')\le K\cdot m,  
\end{equation}
where $K=K(2L)$ is a constant provided by Lemma \ref{lem:a}. Combining inequalities \eqref{eq:main:proof:A} and \eqref{eq:main:proof:B}, we get
\[\frac{\dpn(P,P')}{K\cdot J}-1\le \dwp(X,X').\]
Therefore, the first inequality of \eqref{qi:embedding} is satisfied if one takes $A_2\ge 1$ and $A_1\geq K\cdot J$.
\end{proof}
\begin{proof}[Proof of Theorem \ref{main_pantsN_pantsS}]
By Proposition \ref{J:3Lip}, the map 
\[\map{j}{\pantsNV}{\pantsSgSig}\]
is a bijection. Moreover, for $P,P'\in\pantsNV$ we have
\[ \frac{1}{2DA_1}\cdot \dpn(P,P')-\frac{A_2}{2D}\leq \dps(\J(P),\J(P'))\le 3\cdot \dpn(P,P'),\]
 where $A_1$, $A_2$ are the constants from \eqref{qi:embedding} and $D=D(L)$ is the constant from Proposition \ref{sublevel:prop}.
 The first inequality follows from \eqref{qi:embedding} and the fact that $\Phi$ is $2D$-Lipschitz (see the proof of Lemma~\ref{Lipschitz6d}). The second one is a consequence of Proposition~\ref{J:3Lip}.      
\end{proof}

\end{document}